\documentclass[11pt, a4paper]{amsart}

\usepackage{graphicx}
\usepackage{amsfonts}
\usepackage{amssymb}
\usepackage{amscd}

\usepackage{mathrsfs}
\usepackage{braket}

\usepackage{enumitem}
\setlist[enumerate]{topsep=2pt, parsep=0pt, partopsep=0pt, itemsep=2pt, label=\rm{(\roman*)}}

\makeatletter
\renewcommand{\section}{%
\@startsection{section}{1}%
  \z@{.7\linespacing\@plus\linespacing}{.5\linespacing}%
 {\normalfont\large\bfseries\centering}}
\makeatother

\newtheorem{theorem}{Theorem}[section]
\newtheorem{corollary}[theorem]{Corollary}

\newtheorem{lemma}[theorem]{Lemma}
\newtheorem{proposition}[theorem]{Proposition}

\newtheorem*{theorem*}{Theorem} 
\newtheorem*{corollary*}{Corollary}
\newtheorem*{conjecture*}{Conjecture}
\newtheorem*{lemma*}{Lemma}
\newtheorem*{proposition*}{Proposition}
\newtheorem*{problem*}{Problem}
\newtheorem*{axiom*}{Axiom}
\newtheorem*{example*}{Example}
\newtheorem*{exercise*}{Exercise}

\theoremstyle{definition}

\newtheorem*{remark*}{Remark}

\newtheorem*{definition*}{Definition}

\renewcommand{\l}{\left}
\renewcommand{\r}{\right}
\newcommand{\eps}{\varepsilon}
\newcommand{\N}{{\mathbb N}}
\newcommand{\R}{{\mathbb R}}

\newcommand{\x}{\times}
\newcommand{\del}{\partial}

\def\norm[#1]{\left\Vert #1 \right\Vert}
\def\tbra[#1,#2]{\left\langle #1 , #2\right\rangle} 
\def\rbra[#1,#2]{\left( #1 , #2 \right)} 
\def\sbra[#1,#2]{\left[ #1 , #2 \right]} 

\newcommand{\ce}{\mathrel{\mathop:}=}
\newcommand{\ec}{=\mathrel{\mathop:}}

\def\lang[#1]{\left\langle{#1}\right\rangle}

\newcommand{\sgn}{\operatorname{sgn}}

\begin{document}

\title[]{ Sharp thresholds for stability and instability of standing waves in a double power nonlinear Schr\"odinger equation}


\author{Masayuki Hayashi}
\address{Research Institute for Mathematical Sciences, Kyoto University, Kyoto 606-8502, Japan}
\curraddr{}
\email{hayashi@kurims.kyoto-u.ac.jp}
\thanks{}


\date{}

\dedicatory{}


\begin{abstract}
We study the stability/instability of standing waves for the one dimensional nonlinear Schr\"odinger equation with double power nonlinearities:
\begin{align*}
&i\del_t u +\del_x^2 u -|u|^{p-1}u +|u|^{q-1}u=0, \quad (t,x)\in \R\times\R ,~1<p<q.
\end{align*}
When $q<5$, the stability properties of standing waves $e^{i\omega t}\phi_{\omega}$ may change for the frequency $\omega$. A sufficient condition for yielding instability for small frequencies are obtained in previous results, but it has not been known what the sharp condition is. In this paper we completely calculate the explicit formula of $\lim_{\omega\to0}\del_{\omega}\|\phi_{\omega}\|_{L^2}^2$, which is independent of interest, and establish the sharp thresholds for stability and instability of standing waves.
\end{abstract}

\maketitle

\setcounter{tocdepth}{1}
\tableofcontents

\numberwithin{equation}{section} 
\section{Introduction}

In this paper we consider the double power nonlinear Schr\"{o}dinger equation:
\begin{align}
\label{NLS}
\tag{NLS}
i\del_t u +\Delta u +a|u|^{p-1}u +b|u|^{q-1}u=0, \quad (t,x)\in \R\times\R^d ,~1<p<q.
\end{align}
The energy of \eqref{NLS} is given by
\begin{align*}
E(u)&:=\frac{1}{2}\|\nabla u \|_{L^2}^2 -\frac{a}{p+1}\| u\|_{L^{p+1}}^{p+1} -\frac{b}{q+1}
\| u\|_{L^{q+1}}^{q+1},
\end{align*}
which is conserved under the flow. It is well known that \eqref{NLS} has standing waves $e^{i\omega t}\phi_{\omega}(x)~(\omega>0)$, if the nonlinearities satisfy each of the following conditions:
\begin{enumerate}[label=(\Alph*)]

\item defocusing, focusing $(a>0,\,b<0)$,
\label{df}

\item focusing, defocusing $(a>0,\,b<0)$,
\label{fd}
\item focusing, focusing $(a>0,\,b>0)$.
\label{ff}
\end{enumerate}
From the general theory \cite{GSS87}, the stability/instability of standing waves is determined by the sign of the function
\begin{align*}
M'(\omega)\ce \frac{1}{2}\frac{d}{d\omega}\int_{\R^d} \phi_{\omega}(x)^2dx\quad\text{for}~\omega>0,
\end{align*}
provided the suitable spectral conditions of linearized operators. 
Different from the pure power case, the double power nonlinearities 
destroy the scaling symmetry of the equation, which yields delicate problems to investigate the sign of $M'(\omega)$. It is known as an interesting phenomena in double power nonlinearities that the stability properties may change for the frequency $\omega$ even if $(p,q)$ is fixed (see \cite{O95}).

For the cases \ref{fd}, \ref{ff} in the one-dimensional case, the stability properties are completely determined in \cite{IK93, O95, M08}.
Therefore, we consider the case \ref{df} in $d=1$ here. We note that in the case \ref{df} there exists the standing waves with zero frequency $\omega=0$
as well as $\omega>0$.
By scalar multiplication and the scaling, we may always take 
$a=-1$ and $b=1$ as 
\begin{align}
\label{eq:1.1}
&i\del_t u +\del_x^2 u -|u|^{p-1}u +|u|^{q-1}u=0, \quad (t,x)\in \R\times\R ,~1<p<q.
\end{align}
In this paper we study the stability properties of standing waves $e^{i\omega t}\phi_{\omega}$ of \eqref{eq:1.1}. 
The definition of stability/instability of standing waves is given as follows.
\begin{definition*}
We say that the standing wave $e^{i\omega t}\phi_{\omega}$ of \eqref{eq:1.1} is (orbitally) \emph{stable} if for any $\eps>0$ there exists $\delta>0$ such that the following statement holds:
If $u_0\in H^1(\R)$ satisfies $\|u_0-\phi_{\omega}\|_{H^1}<\delta$, 
then the solution $u(t)$ of \eqref{eq:1.1} exists globally in time and satisfies 
\[
\sup_{t\in\R}\inf_{(\theta,y)\in\R\x\R}\|u(t)-e^{i\theta}\phi_{\omega}(\cdot-y)\|_{H^1}<\eps.
\]
Otherwise, we say that it is (orbitally) \emph{unstable}.
\end{definition*}
In one-dimensional case, the explicit integral formula $M'(\omega)$ was calculated by Iliev and Kirchev \cite{IK93}. Based on the formula in \cite{IK93}, Ohta~\cite{O95} studied stability properties for the case of double power nonlinearities, and proved the following result for \eqref{eq:1.1}:
\begin{itemize}
\item
When $q\ge 5$, then the standing wave $e^{i\omega t}\phi_\omega$ is unstable for all $\omega>0$.
\item
When $q<5$, there exists $\omega_0>0$ such that the standing wave $e^{i\omega t}\phi_\omega$ is stable for $\omega>\omega_0$.
 Assuming further $p+q>6$, then there exists $\omega_1\in (0,\omega_0)$ such that the standing wave $e^{i\omega t}\phi_\omega$ is unstable for $\omega\in (0,\omega_1)$.
\end{itemize}
In short, the stability properties change for the frequency when $q<5$.
Later, Maeda~\cite{M08} improved this result and bridged a gap between $\omega_0$ and $\omega_1$ if $p\ge\frac{7}{3}$, where the assumption of $p$ comes from certain monotonicity of $M''(\omega)$.  
Recently, it was proved in \cite{FH20} that if $\frac{23-3p}{3+p}<q<5$, there exists $\omega_2>0$ that the standing wave is unstable for $\omega\in[0,\omega_2]$.\footnote{Analogous instability results in higher dimensions are also obtained in \cite{FH20}.} We note that if $p\in(1,5)$, then
\begin{align*}
\gamma_1(p)\ce\frac{23-3p}{3+p}<-p+6,
\end{align*}
which yields that the condition $q>\gamma_1(p)$ improves the condition $p+q>6$. On the other hand, it was proved in \cite{O95} that when $p=2,q=3$, the standing wave $e^{i\omega t}\phi_{\omega}$ is stable for all $\omega>0$. Therefore, some condition is necessary to prove the instability for small frequencies when $q<5$, but it has not been known what the sharp condition is.

We recall that the condition $q>\gamma_1(p)$ in \cite{FH20} is characterized in terms of standing waves with zero frequency as
\begin{align*}
q>\gamma_1(p)\iff \l.\del^2_{\lambda}E(\lambda^{\frac{1}{2}}\phi_0(\lambda\cdot)\r|_{\lambda=1}<0.
\end{align*}
Therefore, one can say that this condition was obtained for the first time to focus on the standing wave with zero frequency. 
In the spirit of this observation, we calculate the zero frequency limit of $M'(\omega)$ and derive the sharp condition yielding the instability for small frequencies. 
Although our results in this paper are restricted in one space dimension, 
we can completely calculate the explicit formula of $\lim_{\omega\to0}M'(\omega)$, which is independent of interest.

We now state our main results.
\begin{theorem}
\label{thm:1.1}
Let $q<5$. Then we have
\begin{align}
\label{eq:1.2}
\lim_{\omega\to0}M'(\omega)
=\l\{
\begin{aligned}
&-\infty, &&\text{if}~p\ge\frac{7}{3},
\\
&c_{p,q}\frac{\Gamma\l(\frac{7-3p}{2(q-p)}\r)}{\Gamma\l(\frac{7-2p-q}{2(q-p)}\r)},
&&\text{if}~p<\frac{7}{3},
\end{aligned}
\r.
\end{align}
where $c_{p,q}$ is a positive constant, which is explicitly written as $c_{p,q}=\sqrt{2\pi}\l(\frac{q+1}{p+1}\r)^{\frac{7-3p}{2(q-p)}}\frac{(p+1)^{\frac{3}{2}}}{q-p}$.
\end{theorem}
As an application of Theorem \ref{thm:1.1}, we have the following sharp stability results.
\begin{theorem}
\label{thm:1.2}
Let $q<5$. Then the following statements hold.
\begin{enumerate}
\item If $p\geq\frac{7}{3}$, there exists $\omega_*>0$ such that $e^{i\omega t}\phi_{\omega}$ is unstable if $\omega\in(0,\omega_*]$ and stable if $\omega>\omega_*$.

\item Assume $p<\frac{7}{3}$. If $2p+q>7$, there exists $\mu_1\in(0,\omega_0)$ such that $e^{i\omega t}\phi_{\omega}$ is unstable if $\omega\in(0,\mu_1)$. If $2p+q\le7$, there exists $\mu_2\in(0,\omega_0)$ such that $e^{i\omega t}\phi_{\omega}$ is stable $\omega\in(0,\mu_2)$. 
\end{enumerate}
\end{theorem}
In Theorem \ref{thm:1.2}, there are gaps between $\mu_1,\mu_2$ and $\omega_0$. Inspired from the work of \cite{M08}, we bridge the gaps as follows.
\begin{theorem}
\label{thm:1.3}
Let $q<5$ and $\frac{9}{5}\le p<\frac{7}{3}$. Then, the following statements hold.
\begin{enumerate}
\item If $2p+q> 7$, there exists $\omega_*>0$ such that $e^{i\omega t}\phi_{\omega}$ is unstable if $\omega\in(0,\omega_*]$ and stable if $\omega>\omega_*$.

\item If $2p+q\le7$, $e^{i\omega t}\phi_{\omega}$ is stable for all $\omega>0$.
\end{enumerate}
\end{theorem}
The condition $\frac{9}{5}\le p$ comes from certain monotonicity of $M^{(3)}(\omega)$. We think that this is a technical assumption and the conclusion should hold without the restriction $\frac{9}{5}\le p$, but we do not pursue this issue further here. 

The rest of this paper is organized as follows.
In Section \ref{sec:2} we calculate the zero frequency limit of $M'(\omega)$ and Theorem \ref{thm:1.1}. In Section \ref{sec:3.1} we organize the derivatives of $M(\omega)$ and fundamental properties of zeros and extremal points of these functions. We also give a quick review on previous results \cite{O95, M08} (see Lemma \ref{lem:3.2} below).  
In Sections \ref{sec:3.2} and \ref{sec:3.3}, we prove Theorems \ref{thm:1.2} and \ref{thm:1.3}, respectively. More specifically, based on the results in Section \ref{sec:3.1}, we apply Theorem \ref{thm:1.1} and general theory of \cite{GSS87} to investigate the stability properties of standing waves. 

\section{Zero frequency limit}
\label{sec:2}


We use the integration formula of $M'(\omega)$ by Iliev and Kirchev \cite{IK93}.
\begin{lemma}[\cite{IK93}]
\label{lem:2.1}
For $\omega>0$, we obtain the following formula:
\begin{align}
\label{eq:2.1}
M'(\omega)
=-\frac{1}{4W'(h)}\int_{0}^{h}\frac{K(h)-K(s)}{\l( L(h)-L(s) \r)^{3/2}}ds,
\end{align}
where the functions are defined by
\begin{align*}
K(s)&\ce -\frac{5-p}{p+1}s^{\frac{p-1}{2}}+\frac{5-q}{q+1}s^{\frac{q-1}{2}},\\
L(s)&\ce-\frac{2}{p+1}s^{\frac{p-1}{2}}+\frac{2}{q+1}s^{\frac{q-1}{2}},\\
W(s)&=W(s;\omega)\ce\omega s-L(s)s,
\end{align*}
and $h=h(\omega)$ is a unique positive zero of $W(s;\omega)$.
\end{lemma}
For each $\omega >0$ $h(\omega)$ satisfies 
\begin{align*}
W(h(\omega);\omega)=0,~W'(h(\omega); \omega)<0.
\end{align*}
Therefore, an implicit function theorem yields that $\omega\mapsto h(\omega)$ is a smooth function. We set $h_0\ce h(0)$, which is a positive zero of $L(s)$. We note that
\begin{align*}
W(h(\omega); \omega)=0\iff \omega =L(h(\omega)) \quad\text{for}~\omega\ge0.
\end{align*}
This yields that $\omega \mapsto h(\omega)$ is strictly increasing.


We use parts of notation from \cite{O95, M08} as follows.
\begin{align*}
&c_1=-\frac{5-p}{p+1}, ~c_2=\frac{5-q}{q+1},
~d_1=-\frac{2}{p+1}, ~d_2=\frac{2}{q+1},
\\
&\alpha =\frac{p-1}{2}, ~\beta=\frac{q-1}{2}.
\end{align*}
Then, $K(s)$ and $L(s)$ are rewritten by
\begin{align*}
K(s)&:= c_1s^{\alpha}+c_2s^{\beta},\\
L(s)&:=d_1s^{\alpha}+d_2s^{\beta}.
\end{align*}
We note that $c_1,d_1 <0$ and $c_2, d_2>0$ if $q<5$. Since $h_0$ is a zero of $L(s)$, we have 
\begin{align}
\label{eq:2.2}
h_0^{\beta-\alpha}=-\frac{d_1}{d_2}=\frac{q+1}{p+1}.
\end{align}
We set
\begin{align*}
F(h)\ce\int_{0}^{h}\frac{K(h)-K(s)}{\l( L(h)-L(s) \r)^{3/2}}ds.
\end{align*}
Since $W'(h)<0$, the sign of $M'(\omega)$ coincides with the one of $F(h)$.
Change the variable $s\mapsto hs$, we have
\begin{align*}
F(h)&=h\int_{0}^1 \frac{K(h)-K(hs)}{\l( L(h)-L(hs) \r)^{3/2}}ds
\\
&=h^{1-\frac{\alpha}{2}}\int_{0}^{1}\frac{ c_1(1-s^{\alpha})+c_2(1-s^{\beta})h^{\beta -\alpha}  }{\l( d_1(1-s^{\alpha})+d_2(1-s^{\beta})h^{\beta -\alpha} \r)^{3/2}}ds.
\end{align*}
The zero frequency limit corresponds to the limit $h\to h_0$. 
A direct calculation shows that
\begin{align*}
\lim_{\omega\to0}M'(\omega)&=-\frac{1}{4W'(h_0)}F(h_0)
\\
&=-\frac{h_0^{1-\frac{\alpha}{2}} }{4W'(h_0)}(p+1)^{1/2}2^{-1/2}
\int_0^1
\frac{ -(2-\alpha)(1-s^\alpha)+(2-\beta)(1-s^\beta) }{ (s^\alpha-s^\beta )^{3/2}}
ds.
\end{align*}
%
Here we set 
\begin{align*}
H(\alpha,\beta)\ce 
\int_0^1\frac{ -(2-\alpha)(1-s^\alpha)+(2-\beta)(1-s^\beta) }{ (s^\alpha-s^\beta )^{3/2}}
ds.
\end{align*}
The denominator is rewritten as
\begin{align*}
(s^\alpha-s^\beta )^{3/2}=s^{\frac{3}{2}\alpha}(1-s^{\beta-\alpha})^{3/2},
\end{align*}
Hence the singularity of the origin is like $s^{-\frac{3}{2}\alpha}$.
We note that
$\frac{3}{2}\alpha\ge 1 \iff p\ge\frac{7}{3}$.
Therefore we obtain that
\begin{align}
\label{eq:2.3}
\begin{alignedat}{3}
&p\ge\frac{7}{3}&&\implies &&\lim_{\omega\to0}M'(\omega)=-\infty,
\\[3pt]
&p<\frac{7}{3}&&\implies &&\lim_{\omega\to0}M'(\omega)\in\R.
\end{alignedat}
\end{align}
We now calculate the explicit value of $\lim_{\omega\to0}M'(\omega)$ when $p<\frac{7}{3}$. Theorem \ref{thm:1.1} follows from the following proposition.
\begin{proposition}
\label{prop:2.2}
Let $0<\alpha <\beta <2$ and $0<\alpha <\frac{2}{3}$. Then we have
\begin{align}
\label{eq:2.4}
H(\alpha ,\beta) =2\sqrt{\pi}\,
\frac{ \Gamma\l( \frac{-3\alpha +2}{2(\beta -\alpha)}\r) }
{\Gamma\l( \frac{2-2\alpha-\beta}{2(\beta-\alpha)}\r)}.
\end{align}
\end{proposition}
\begin{proof}
First we change variables $t=s^{\beta-\alpha}$ to obtain
\begin{align*}
H(\alpha,\beta)
&=\int_{0}^{1}\frac{-(2-\alpha)(1-t^{\frac{\alpha}{\beta-\alpha}})+(2-\beta)(1-t^{\frac{\beta}{\beta-\alpha}}) }{t^{\frac{3\alpha}{2(\beta-\alpha)}}(1-t)^{\frac{3}{2}} }\cdot \frac{dt}{ (\beta-\alpha)t^{1-\frac{1}{\beta-\alpha}} }
\\
&=\frac{1}{\beta-\alpha}\int_{0}^{1}t^{\delta-1}(1-t)^{-\frac{3}{2}}
\l[-(2-\alpha)(1-t^\gamma)+(2-\beta)(1-t^{\gamma+1})\r]dt,
\end{align*}
where $\gamma\ce\frac{\alpha}{\beta-\alpha},\,\delta\ce\frac{2-3\alpha}{2(\beta-\alpha)}$.
We set $f(t)\ce1-t^{\gamma}$. By the Taylor expansion around $t=1$, we have
\begin{align}
\label{eq:2.5}
f(t)=\sum_{n=1}^{\infty}\frac{f^{(n)}(1)}{n!}(t-1)^n=-\sum_{n=1}^{\infty}\frac{(-\gamma)_n}{n!}(1-t)^n
\quad\forall t\in(0,1),
\end{align}
where $(-\gamma)_n$ is the Pochhammaer symbol defined by
\begin{align*}
(-\gamma)_n&=(-\gamma)(-\gamma+1)\cdots(-\gamma+(n-1))
\quad\text{for}~n\in\N,
\\
(-\gamma)_0&=1.
\end{align*}
By substituting \eqref{eq:2.5} into the integrands, we obtain
\begin{align*}
(\beta-\alpha)H(\alpha,\beta) 
&=\sum_{n=1}^{\infty}\frac{(2-\alpha)(-\gamma)_n-(2-\beta)(-\gamma-1)_n}{n!}\int_{0}^{1}t^{\delta-1}(1-t)^{n-\frac{3}{2}}dt
\\
&=\sum_{n=1}^{\infty}\frac{(2-\alpha)(-\gamma)_n-(2-\beta)(-\gamma-1)_n}{n!}B\l(\delta,n-\tfrac{1}{2}\r),
\end{align*}
where $B\l(\delta,n-\tfrac{1}{2}\r)$ is the beta function, which is rewritten by using Gamma functions as
\begin{align*}
B\l(\delta,n-\tfrac{1}{2}\r)=\frac{\Gamma(\delta)\Gamma(n-\tfrac{1}{2})}{\Gamma(n+\delta-\tfrac{1}{2})}=\frac{(-\tfrac12)_n}{(\delta-\tfrac12)_n}
\cdot\frac{\Gamma(\delta)\Gamma(-\tfrac12)}{\Gamma(\delta-\tfrac12)}.
\end{align*}
Then, we have
\begin{align*}
(\beta-\alpha)H(\alpha,\beta) &=
\frac{\Gamma(\delta)\Gamma(-\tfrac12)}{\Gamma(\delta-\tfrac12)}
\l(
(2-\alpha)\sum_{n=1}^{\infty}\frac{(-\gamma)_n(-\tfrac12)_n}{(\delta-\tfrac12)_n} 
-(2-\beta)\sum_{n=1}^{\infty}\frac{(-\gamma-1)_n(-\tfrac12)_n}{(\delta-\tfrac12)_n} 
\r)
\\
&=
\frac{\Gamma(\delta)\Gamma(-\tfrac12)}{\Gamma(\delta-\tfrac12)}
\begin{aligned}[t]
\Bigl((2-\alpha)F(-\gamma,-\tfrac{1}{2},\delta-\tfrac{1}{2};1)
&-(2-\beta)F(-\gamma-1,-\tfrac{1}{2},\delta-\tfrac{1}{2};1)
\\&~-(2-\alpha)+(2-\beta) \Bigr),
\end{aligned}
\end{align*}
where $F(a,b,c;z)$ is the Gauss hypergeometric function
\begin{align*}
F(a,b,c;z) \ce\sum_{n=0}^{\infty} \frac{(a)_n(b)_n}{(c)_n}\cdot\frac{z^n}{n!} \quad\text{for}~|z|<1.
\end{align*}
When $c>a+b$, the series absolutely converges for $|z|=1$ (see \cite[9.102]{GR07}). 
In our case this condition is satisfied because
\begin{align*}
\delta-\tfrac{1}{2}>(-\gamma)+\l(-\tfrac{1}{2}\r)\iff \delta+\gamma >0.
\end{align*}
We now use the following recursion formula [9.137, 2.]\cite{GR07}:
\begin{align*}
(2a-c-az+bz)F(a,b,c;z)+(c-a)F(a-1,b,c;z)+a(z-1)F(a+1,b,c;z)=0
\end{align*}
with $z=1$:
\begin{align}
\label{eq:2.6}
(a+b-c)F(a,b,c;1)+(c-a)F(a-1,b,c;1)=0.
\end{align}
If we set
\begin{align*}
a=-\gamma,~b=-\frac12,~c=\delta-\frac12,
\end{align*}
then we have
\begin{align*}
a+b-c&=-\frac{2-\alpha}{2(\beta-\alpha)},
\\
c-a&=\frac{2-\beta}{2(\beta-\alpha)}.
\end{align*}
Therefore, it follows from \eqref{eq:2.6} that
\begin{align*}
(2-\alpha)F(-\gamma,-\tfrac{1}{2},\delta-\tfrac{1}{2};1)
-(2-\beta)F(-\gamma-1,-\tfrac{1}{2},\delta-\tfrac{1}{2};1)=0.
\end{align*}
Substituting this relation into the above formula, we obtain
\begin{align*}
H(\alpha,\beta)
=-\frac{\Gamma(\delta)\Gamma(-\tfrac12)}{\Gamma(\delta-\tfrac12)}.
\end{align*}
The relation \eqref{eq:2.4} is obtained from $ \Gamma \l( -\tfrac{1}{2}\r)=-2\sqrt{\pi}$ and the definition of $\delta$.
\end{proof}
As a simple corollary of Proposition~\ref{prop:2.2}, we have the following.
\begin{corollary}
\label{cor:2.3}
Let $1<p<q<5$ and $1<p<\frac{7}{3}$. Then we have
\begin{align}
\label{eq:2.7}
\lim_{\omega\downarrow0}\sgn M'(\omega)
=\l\{
\begin{alignedat}{2}
&{-1} &\qquad&\text{if}\quad2p+q>7,
\\
&0 &&\text{if}\quad2p+q=7,
\\
&1 &&\text{if}\quad 2p+q<7.
\end{alignedat}
\r.
\end{align}
\end{corollary}


\section{Stability/instability on middle frequencies}
\label{sec:3}

In this section we study stability/instability of standing waves on middle frequencies. To this end it is useful to take advantage of higher derivatives of $M(\omega)$. 
\subsection{Higher derivatives of $M(\omega)$}
\label{sec:3.1}
We recall the integral formula of $F(h)$:
\begin{align*}
F(h)&=h^{1-\frac{\alpha}{2}}\int_{0}^{1}\frac{ c_1(1-s^{\alpha})+c_2(1-s^{\beta})h^{\beta -\alpha}  }{\l( d_1(1-s^{\alpha})+d_2(1-s^{\beta})h^{\beta -\alpha} \r)^{3/2}}ds
\ec h^{1-\frac{\alpha}{2}}F_0(h).
\end{align*}
Here we note that 
\begin{align}
\label{eq:3.1}
\sgn M'(\omega)=\sgn F(h)=\sgn F_0(h) \quad\text{for}~h=h(\omega).
\end{align}
Therefore, to check the sign of $M'(\omega)$ is reduced to investigate the sign of $F_0(h)$.
By a direct calculation we have
\begin{align*}
F_0'(h)&=\frac{d_2}{2}(\beta-\alpha)h^{\beta-\alpha-1}
\int_{0}^{1}\frac{ (1-s^{\beta})\l(-r_1(1-s^{\alpha})-c_2(1-s^{\beta})h^{\beta -\alpha}\r)  }{\l( d_1(1-s^{\alpha})+d_2(1-s^{\beta})h^{\beta -\alpha} \r)^{5/2}}ds
\\
&\ec \frac{d_2}{2}(\beta-\alpha)h^{\beta-\alpha-1}F_1(h),
\\[3pt]
F_1'(h)&=\frac{3d_2}{2}(\beta-\alpha)h^{\beta-\alpha-1}
\int_{0}^{1}\frac{ (1-s^{\beta})^2\l(r_2(1-s^{\alpha})+c_2(1-s^{\beta})h^{\beta -\alpha}\r)  }{\l( d_1(1-s^{\alpha})+d_2(1-s^{\beta})h^{\beta -\alpha} \r)^{7/2}}ds
\\
&\ec \frac{3d_2}{2}(\beta-\alpha)h^{\beta-\alpha-1}F_2(h),
\end{align*}
where 
\begin{align*}
r_1&\ce c_1+d_1(q-p),~r_2\ce c_1+2d_1(q-p).
\end{align*}
We note that $c_1,d_1,r_1,r_2<0$ and $c_2,d_2>0$ if $q<5$.
To sum up, we have the following.
\begin{lemma}
\label{lem:3.1}
Each function $(h_0,\infty)\mapsto F_j(h)\,(j=0,1,2)$ is differentiable and
\begin{align}
F(h)&=h^{\frac{2-\alpha}{2}}F_0(h),
\label{eq:3.2}\\
F_0'(h)&=\frac{d_2}{2}(\beta-\alpha)h^{\beta-\alpha-1}F_1(h),
\label{eq:3.3}\\
F_1'(h)&=\frac{3d_2}{2}(\beta-\alpha)h^{\beta-\alpha-1}F_2(h).
\label{eq:3.4}
\end{align}
Moreover, the function $F_j(h)\,(j=0,1,2)$ is represented as
\begin{align}
F_0(h)&=h^{-1+\frac{\alpha}{2}}\int_{0}^{h}\frac{K(h)-K(s)}{\l( L(h)-L(s) \r)^{3/2}}ds,
\label{eq:3.5}
\\
h^{\beta-\alpha}F_1(h)&=h^{-1+\frac{\alpha}{2}}\int_{0}^{h} \frac{ (h^{\beta}-s^{\beta})(K_1(h)-K_1(s)) }{ \l( L(h)-L(s)\r)^{5/2} }ds,
\label{eq:3.6}
\\
h^{2(\beta-\alpha)}F_2(h)&=h^{-1+\frac{\alpha}{2}}
\int_{0}^{h}\frac{ (h^{\beta}-s^{\beta})^2(K_2(h)-K_2(s)) }{ \l( L(h)-L(s)\r)^{5/2} }ds,
\label{eq:3.7}
\end{align}
where the functions in the integrands are defined by
\begin{align*}
K_1(s)&=-r_1s^{\alpha}-c_2^2s^{\beta},~K_2(s)=r_2s^{\alpha}+c_2^2s^{\beta}.
\end{align*}
\end{lemma}
For $j=0,1,2$ we denote a unique positive zero of $K_j$ by $s_j$, and a unique positive extremal point of  $K_j$ by $t_j$, which are explicitly represented as 
\begin{align*}
&s_0^{\beta-\alpha} =-\frac{c_1}{c_2},~s_1^{\beta-\alpha}=-\frac{r_1}{c_2},
~s_2^{\beta-\alpha} =-\frac{r_2}{c_2},\\
&t_0^{\beta-\alpha} =-\frac{c_1\alpha}{c_2\beta},~t_1^{\beta-\alpha}=-\frac{r_1\alpha}{c_2\beta},~t_2^{\beta-\alpha} =-\frac{r_2\alpha}{c_2\beta}.
\end{align*}
From the formula \eqref{eq:3.5}, we deduce that
\begin{align}
\label{eq:3.8}
\sgn F_0(h)=1\quad\text{for}~h\ge s_0.
\end{align}
We note that $\omega_0$ in the introduction is determined by $\omega_0=L(s_0)$.

By a simple calculation, we obtain the following relations.
\begin{lemma}
\label{lem:3.2}
Let $1<p<q<5$. We have 
\begin{enumerate}
\item $h_0<t_0\iff p+q>6$.
\label{pq6}

\item $h_0<t_1\iff q>-3p+8$.
\label{h0t1}

\item $s_0\le t_1\iff \frac{7}{3}\le p$.
\label{73p}

\item $s_0\le t_2\iff \frac{9}{5}\le p$.
\label{95p}
\end{enumerate}
\end{lemma}
Let us give some some comments on Lemma \ref{lem:3.2}. 
It follows from \eqref{eq:3.5} that $F_0(h)$ is negative if $h\in (h_0,t_0)$, which is possible if $p+q>6$ from the assertion \ref{pq6}. The condition in \cite{O95} was derived in this way.
The assertion \ref{h0t1} is used to obtain the stability result for the threshold case $2p+q=7$ later. 
It follows from the assertion \ref{73p} that if $p\ge\frac{7}{3}$, $\sgn F_0'(h)=\sgn F_1(h)$ for $h\in (h_0,s_0)$, which was used in \cite{M08} to determine the stability/instability on middle frequencies. When $p\ge\frac{9}{5}$, it follows from the assertion \ref{95p} that $\sgn F_1'(h)=\sgn F_2(h)$ for $h\in (h_0,s_0)$, which is newly used in this paper.

\subsection{Proof of Theorem \ref{thm:1.2}}
\label{sec:3.2}
As an application of Theorem \ref{thm:1.1}, we first prove Theorem \ref{thm:1.2}. We use the following stability/instability criterion.
\begin{lemma}[\cite{GSS87}]
\label{lem:3.3}
Let $\omega>0$. Then, the standing wave $e^{i\omega t}\phi_{\omega}(x)$ is stable if $M'(\omega)>0$, and unstable if $M'(\omega)<0$. 
\end{lemma}
\begin{proof}[Proof of Theorem \ref{thm:1.2}]
(i) It follows from \eqref{eq:2.3} that there exists $\mu_0>0$ such that $M'(\omega)<0$ for any $\omega\in (0,\mu_0)$. By Lemma \ref{lem:3.2} \ref{73p} $F_0'>0$ on $(h_0,s_0)$. Therefore, there exists a unique $z_*\in (h_0,s_0)$ such that 
\begin{align}
\label{eq:3.9}
F_0(h)<0~\text{on}~(h_0,z_*),~F_0(z_*)=0,~F_0(h)>0~\text{on}~(z_*,\infty) .
\end{align}
We set $\omega_*\ce L(z_*)$. Since $\sgn M'(\omega)=\sgn F_0(h(\omega))$, it follows from Lemma \ref{lem:3.3} that $e^{i\omega t}\phi_{\omega}(x)$ is stable if $\omega\in(0,\omega_*)$, and unstable if $\omega>\omega_*$.

We now consider the remaining case $\omega=\omega_*$. From \eqref{eq:3.9} we have
\begin{align*}
F_0(z_*)=0,~F_0'(z_*)>0.
\end{align*}
We use the formula of $M'(\omega)$: 
\begin{align*}
M'(\omega)=-\frac{F(h)}{4W'(h)}
=-\frac{h^{\frac{2-\alpha}{2}} }{4W'(h)}F_0(h),
\end{align*}
and note that
\begin{align*}
\l.\frac{d}{dh}\l(-\frac{h^{\frac{2-\alpha}{2}} }{4W'(h)}F_0(h) \r)\r|_{h=z_*}
=-\frac{z_*^{\frac{2-\alpha}{2}} }{4W'(z_*)}F_0'(z_*)>0.
\end{align*}
Therefore, we obtain that
\begin{align*}
M''(\omega_*)=\l.\frac{d}{dh}\l(-\frac{h^{\frac{2-\alpha}{2}} }{4W'(h)}F_0(h) \r)\r|_{h=z_*}\frac{dh}{d\omega}(\omega_*)>0.
\end{align*}
Then, by applying instability theory \cite{CP03, O11, M12} for the degenerate case, we deduce that $e^{i\omega_*t}\phi_{\omega_*}$ is unstable.
\\[5pt]
(ii) When $2p+q\lessgtr7$, the result follows from Corollary \ref{cor:2.3} and Lemma \ref{lem:3.3}. We now consider the threshold case $2p+q=7$. In this case it follows from Lemma \ref{lem:3.2} \ref{h0t1} that $h_0<t_1$. 
From \eqref{eq:3.3} and \eqref{eq:3.6} we have $F_0'(h)>0$ for $h\in(h_0,t_1)$. From \eqref{eq:2.7} and \eqref{eq:3.1} we deduce that 
$\sgn F_0(h)>0$ for for $h\in(h_0,t_1)$. Hence, the stability result follows from Lemma \ref{lem:3.3}.
\end{proof}
\subsection{Proof of Theorem \ref{thm:1.3}}
\label{sec:3.3}
In what follows we assume that $\frac{9}{5}\le p<\frac{7}{3}$. In this case we have 
\begin{align*}
t_0<t_1<s_0\le t_2 <s_1<s_2.
\end{align*}
We note that $K_2$ is strictly decreasing on $(0,t_2)$, so it follows from the formula \eqref{eq:3.7} that
\begin{align}
\label{eq:3.10}
F_2(h)<0\quad\text{for}~h\in (h_0,s_0).
\end{align}

We now prepare a few lemmas to prove Theorem \ref{thm:1.3}.
\begin{lemma}
\label{lem:3.4}
There exists a small $\eps>0$ such that 
\begin{align}
\label{eq:3.11}
\sgn F_0(h)=
\l\{
\begin{alignedat}{2}
&{-1} &\qquad&\text{if}\quad2p+q>7,
\\
&~1 &&\text{if}\quad 2p+q\le7
\end{alignedat}
\r.
\end{align}
for any $h\in(h_0,h_0+\eps)$.
\end{lemma}
\begin{proof}
The claim follows from the proof of Theorem \ref{thm:1.2}\,(ii).
\end{proof}
\begin{lemma}
\label{lem:3.5}
$F_0'(h_0+0)=\infty$.
\end{lemma}
\begin{proof}
From the formula of $F_1$, we have
\begin{align*}
F_1(h_0)=2^{-5/2}(p+1)^{3/2}\int_0^1\frac{ (1-s^{\beta})\bigl( (5-p+2(q-p))(1-s^{\alpha})-(5-q)(1-s^\beta) \bigr) }{ (s^{\alpha}-s^{\beta})^{5/2}}ds.
\end{align*}
In the integrand the singularity of the origin is like $s^{-\frac{5}{2}\alpha}$.
We note that
\begin{align*}
\frac{5}{2}\alpha\ge 1\iff p\ge \frac{9}{5},
\end{align*}
so the integral above diverges in this case.
Combined with the fact
\begin{align*}
5-p+2(q-p)-(5-q)=3(q-p)>0,
\end{align*}
we deduce that $F_1(h_0)=\infty$. Hence, the conclusion follows from \eqref{eq:3.3}.
\end{proof}
\begin{lemma}
\label{lem:3.6}
There exists at most one zero of $F_0'(h)$ on $(h_0,s_0)$. If the zero exists, which we denote by $z_0$, then we have
\begin{align}
\label{eq:3.12}
\begin{aligned}
&F_0'(h)>0&&\text{if}~h\in (h_0,z_0),
\\
&F_0'(h)<0&&\text{if}~h\in (z_0,s_0).
\end{aligned}
\end{align}
\end{lemma}
\begin{proof}
If we assume $F_1$ has a zero $z_0$ on $(h_0,s_0)$, then it follows from \eqref{eq:3.4} and \eqref{eq:3.10} that
\begin{align}
\label{eq:3.13}
F_1'(z_0)=\frac{3d_2}{2}(\beta-\alpha)z_0^{\beta-\alpha-1}F_2(z_0)<0.
\end{align}
This yields that the number of zeros of $F_1$ on $(h_0,s_0)$ is at most one. The last assertion follows from \eqref{eq:3.13} and $\sgn F_0'=\sgn F_1 >0$ on 
$(h_0,h_0+\eps)$ for some small $\eps>0$.
\end{proof}
\begin{lemma}
\label{lem:3.7}
If we assume $F_0(h_*)>0$ for some $h_*\in(h_0,s_0)$, then we have $F_0(h)>0~\text{on}~(h_*,\infty)$.
\end{lemma}
\begin{proof}
If there is no zero of $F_0'$ on $(h_0,s_0)$, then $F_0'(h)>0$ on 
$(h_0,s_0)$ and the claim follows from \eqref{eq:3.8}. Now we consider the case that there exists a zero $z_0$ of $F_0'$ on $(h_0,s_0)$.
If we assume that $F_0$ has a zero $z_1\in (h_*,s_0)$, then $F_0'(z_1)\le0$. From \eqref{eq:3.12}, we obtain that $z_0\le z_1<s_0$. Therefore, we deduce that
\begin{align*}
F_0(z_1)=0,\quad F_0'(h)<0~\text{on}~(z_1,s_0)
\end{align*}
which contradicts $F_0(s_0)>0$. Hence $F_0$ has no zero $(h_*,s_0)$, the conclusion follows from \eqref{eq:3.8}.
\end{proof}
The proof of Theorem \ref{thm:1.3} is reduced to prove the following claim.
\begin{proposition}
\label{prop:3.8}
The following statements hold.
\begin{enumerate}
\item If $2p+q>7$, there exists $z_*\in(0,s_0)$ such that
\begin{align*}
F_0(h)<0~\text{on}~(h_0,z_*),~ F_0(z_*)=0,~F_0(h)>0~\text{on}~(z_*,\infty).
\end{align*}

\item If $2p+q\le7$, $F_0(h)>0~\text{on}~(h_0,\infty)$.
\end{enumerate}
\end{proposition}
\begin{proof}
If there is no zero of $F_0'$ on $(h_0,s_0)$, 
the claim follows from Lemmas~\ref{lem:3.4}, \ref{lem:3.5} and \eqref{eq:3.8}.

Now we consider the case that there exists a zero $z_0$ of $F_0'$ on $(h_0,s_0)$. From Lemma~\ref{lem:3.6}, $F_0'$ satisfies \eqref{eq:3.12}.
If $2p+q\ge 7$, it follows from Lemma \ref{lem:3.4} that $F_0>0~\text{on}~(h_0,h_0+\eps)$. Combined with Lemma~\ref{lem:3.7}, we deduce that $F_0>0~\text{on}~(h_0,\infty)$. 
If $2p+q<7$, it follows from Lemma~\ref{lem:3.4} and \eqref{eq:3.8} that
there exists $z_*\in(h_0,s_0)$ satisfying
\begin{align*}
F_0(h)<0~\text{on}~(h_0,z_*),~\text{and}~F_0(z_*)=0,
\end{align*}
which yields that $F_0'(z_*)\ge0$. If $F_0'(z_*)=0$, it follows from Lemma \ref{lem:3.6} that $z_*=z_0$. From \eqref{eq:3.12}, we have
\begin{align*}
F_0(z_*)=0,\quad F_0'(h)<0~\text{on}~(z_*,s_0).
\end{align*}
which contradicts $F_0(s_0)>0$. Hence $F_0'(z_*)>0$. This yields that$F_0>0~\text{on}~(z_*,z_*+\eps)$ for small $\eps>0$. Combined with Lemma~\ref{lem:3.7}, we deduce that $F_0>0~\text{on}~(z_*,\infty)$. This completes the proof.
\end{proof}
The proof of Theorem \ref{thm:1.3} from Proposition \ref{prop:3.8} is done in the same way as the proof of Theorem \ref{thm:1.2}. We omit the details.

\section*{Acknowledgments}
This work was supported by JSPS KAKENHI Grant Number JP19J01504.


\end{document}